\documentclass[reqno]{amsart}
\usepackage{hyperref}
\usepackage{setspace}

\def \R{\mathbb{R}}

\def \N{\mathbb{N}}
\def \E{\mathbb{E}}

\def \bf{\textbf}

\def \ni {\noindent}

\vspace{9mm}

\begin{document}
\setcounter{page}{1}

\mbox{}
\title[Controllability of Time-dependent  NSFDE driven by fBm]{ Controllability
of Time-dependent   Neutral   Stochastic Functional Differential
Equations Driven By
 A Fractional  Brownian Motion }
 \bigskip

\bigskip

\author[E.  LAKHEL]{  El Hassan Lakhel }
\maketitle
\begin{center}{National School of Applied Sciences, Cadi Ayyad University, 46000
Safi,\\ Morocco}\end{center}

\vspace{0.8cm}

 \footnotetext[1]{Lakhel E.: e.lakhel@uca.ma      }

\medskip

\begin{abstract}
In this paper we consider the controllability of certain  class of
non-autonomous    neutral evolution stochastic functional
differential equations, with time varying  delays, driven   by a
fractional Brownian motion  in a separable real Hilbert space.
Sufficient conditions for controllability are obtained  by employing
a fixed point approach. A practical example is provided to
illustrate the viability of the abstract result of this work.
\end{abstract}

\medskip
 \ni {\bf {Keywords:}} Controllability.  Neutral stochastic functional differential
equations. Evolution operator. Fractional Brownian motion.

 \maketitle \numberwithin{equation}{section}
\newtheorem{theorem}{Theorem}[section]
\newtheorem{lemma}[theorem]{Lemma}
\newtheorem{proposition}[theorem]{Proposition}
\newtheorem{definition}[theorem]{Definition}
\newtheorem{example}[theorem]{Example}
\newtheorem{remark}[theorem]{Remark}
\allowdisplaybreaks


\section{Introduction}
Controllability is one of the fundamental concepts in mathematical
control theory and plays an important role in control systems.
Controllability generally means that it is possible to steer a
dynamical control system from an arbitrary initial state to an
arbitrary final state using the set of admissible controls. If the
system cannot be controlled completely then different types of
controllability can be defined such as approximate, null, local
null and local approximate null controllability. A standard
approach is to transform the controllability problem into a
fixed-point problem for an appropriate operator in a functional
space. The problem of controllability for functional differential
systems has been extensively studied in many papers
\cite{arthi11,balach10,balach11,chang09,mahmod03,shak11}. For
example, Sakthivel and Ren \cite{Sakren11} studied the
     complete controllability of stochastic evolution equations with jumps. In \cite{balad02}, Balasubramaniam and
Dauer   discussed the controllability of semilinear stochastic delay
evolution equations in Hilbert spaces.

It is known that fractional Brownian motion, with Hurst parameter
$H\in(0,1)$,
  is a generalization of Brownian motion, it reduces to Brownian motion when
  $H=\frac{1}{2}$. A general theory for the infinite dimensional stochastic differential equations   driven by
a fractional Brownian motion (fBm) is not yet established and just
a few results have been proved.
 In addition, in many mathematical models the claims often
display long-range memories, possibly due to extreme weather,
natural disasters, in some cases, many stochastic dynamical systems
depend not only on present and past states, but also contain the
derivatives with delays. Neutral functional differential equations
are often used to describe such systems. Very recently, neutral
stochastic functional differential equations driven by fractional
Brownian motion   have attracted the interest of many researchers.
One can see \cite{boufoussi3, carab, lak15} and the references
therein. The literature concerning the existence and qualitative
properties of solutions of time-dependent functional stochastic
differential equations  is very restricted and limited to a very few
articles. This fact is the main motivation of our  work. We mention
here the recent paper by  Ren et al.   \cite{ren13} concerning the
existence of mild solutions for a class  of stochastic evolution
equations driven by fractional Brownian  motion in Hilbert space.
    \\
Motivated by the above works, this paper is concerned with  the
controllability results for a class of time-dependent
 neutral functional stochastic differential equations  described in the form:
 {\begin{small}\begin{equation}\label{eq1}
 \left\{\begin{array}{llll}
d[x(t)+g(t,x(t-r(t)))]=[A(t)x(t)+f(t,x(t-\rho(t)))+Bu(t)]dt+\sigma
(t)dB^H(t),\, t \in[0, T],  \\ \\

x(t)=\varphi(t) ,\;-\tau \leq t \leq 0\quad a.s.\quad \tau>0,
\end{array}\right.
\end{equation}\end{small}

in a real Hilbert space $X$ with inner product $<.,.>$ and norm
$\|.\|$,  where \{ $A(t),\,t\in[0,T]\}$ is a family of linear
closed operators from a  space $X$ into $X$ that generates an
  evolution system of operators $\{U(t,s),\, 0\leq s\leq t \leq T\}$. $B^H$ is a fractional Brownian motion on a
real and separable Hilbert space $Y$, $r,\; \rho
\;:[0,+\infty)\rightarrow [0,\tau]\; (\tau
>0)$ are continuous and $f,g:[0,+\infty)\times X \rightarrow X,\;
\; \sigma:[0,+\infty) \rightarrow \mathcal{L}_2^0(Y,X)$,$\; \;$
   are appropriate
functions. Here $\mathcal{L}_2^0(Y,X)$ denotes the space of all
$Q$-Hilbert-Schmidt operators from $Y$ into $X$ (see section 2
below).

To the best of our knowledge, there is no paper which investigates
the study of controllability for  time-dependent neutral stochastic
functional differential equations with delays driven   by fractional
Brownian motion. Thus, we will make the first attempt
to study such problem in this paper.\\
 Our results are inspired by the one in
\cite{boufoussi14} where the existence and uniqueness of mild
solutions to  model (\ref{eq1}) with $B=0,$   is studied.

The  rest of this  paper is organized as follows.  Section 2,
recapitulate some notations,  basic  concepts, and basic results
about fractional Brownian motion,  Wiener integral over Hilbert
spaces and we recall some preliminary results about evolution
operator.  Section 3, gives sufficient conditions to prove the
controllability result for the problem $(\ref{eq1})$. In Section 4
we give an example to illustrate the efficiency of the obtained
result.
\section{Preliminaries}\label{sec:1}
\subsection{Evolution families}
In this subsection we introduce the notion of evolution family.

\begin{definition}\label{d1}  A set  $\{U(t, s):  0\leq s \leq t\leq T\}$
of bounded linear operators on a Hilbert space $X$ is called an
\emph{evolution family} if
\begin{itemize}
\item[(a)] $U(t,s)U(s,r)=U(t,r)$, $U(s, s)=I$ if $ r \leq s \leq  t
$,
\item[(b)] $(t,s)\to U(t,s)x$ is strongly continuous for $t> s$.
\end{itemize}
\end{definition}
Let $\{A(t),\, t\in[0,T]\}$ be a family
of closed densely defined linear unbounded operators on the
Hilbert space $X$ and with domain $D(A(t))$ independent of $t$, satisfying the following conditions introduced by   \cite{AT}.

There exist constants $\lambda_0\geq 0$, $\theta\in
(\frac{\pi}{2}, \pi)$, $L$,
 $K\geq 0$, and $\mu$, $\nu\in (0, 1]$ with $\mu +\nu >1$ such that
\begin{equation}\label{AT1}
\Sigma_\theta\cup \{0\}\subset\rho (A(t)-\lambda_0), \quad
\|R(\lambda, A(t)-\lambda_0)\|\leq\frac{K}{1+|\lambda|}
\end{equation}
and
\begin{equation}\label{AT2}
\|(A(t)-\lambda_0)R(\lambda, A(t)-\lambda_0)\big[R(\lambda_0,
A(t)) -R(\lambda_0, A(s))\big]\|\leq L|t-s|^\mu|\lambda|^{-\nu},
\end{equation}
for $t$, $s\in\mathbb{R}$,
$\lambda\in\Sigma_\theta$ where $\Sigma_\theta:=\big\{\lambda\in{\mathbb{C}}-\{0\}:
 |\arg \lambda|\leq\theta\big\}$.

It is well known, that this assumption implies that there exists a unique evolution
family $\{U(t, s):  0\leq s \leq t\leq T\}$ on $X$ such that $(t,
s)\to U(t, s)\in{\mathcal L}(X)$ is continuous for $t>s$,
$U(\cdot, s)\in \mathcal{C}^1((s, \infty), {\mathcal L}(X))$,
 $\partial_tU(t, s)=A(t)U(t, s)$, and
\begin{equation}\label{w1}
\|A(t)^kU(t, s)\|\leq C(t-s)^{-k}
\end{equation}
for $0<t-s\leq 1$, $k=0, 1$, $0\leq \alpha <\mu$, $x\in
D((\lambda_0-A(s))^\alpha)$, and a constant $C$ depending only on
the constants in (\ref{AT1})-(\ref{AT2}). Moreover,
$\partial_s^+U(t, s)x=-U(t, s)A(s)x$ for $t>s$ and $x\in D(A(s))$
with $A(s)x\in \overline {D(A(s))}$. We say that $A(\cdot)$
generates $\{U(t, s): 0\leq s \leq t\leq T\}$. Note that $U(t,s)$
is exponentially bounded by (\ref{w1}) with $k=0$.

\begin{remark} If $\{A(t), \,t\in[0,T]\}$ is a second order
differential operator $A$, that is $A(t)=A$ for each $t\in[0,T]$,
then $A$ generates a $C_0-$semigroup $\{e^{At}, t\in[0,T]\}$.
\end{remark}

For additional details on evolution system and their properties,
we refer the reader to \cite{pazy}.

\subsection{Fractional Brownian Motion}

For the convenience for the reader we recall briefly here some of
the basic results of fractional Brownian motion calculus. For
details of this section, we refer the reader to \cite{nualart} and
the references therein.\\

Let $(\Omega,\mathcal{F}, \mathbb{P})$ be a complete probability
space. A standard fractional Brownian motion (fBm) $\{\beta^H(t),
t\in\mathbb{R}\}$
 with Hurst parameter $H\in (0, 1)$ is a zero mean  Gaussian process with continuous
sample paths such that
\begin{eqnarray}\label{A3}
\mathbb{E}[\beta^H(t)\beta^H(s)]=\frac{1}{2}\big(t^{2H}+s^{2H}-|t-s|^{2H}\big)
\end{eqnarray}
for $s$, $t\in\mathbb{R}$. It is clear that for $H=1/2$, this
process is a standard Brownian motion. In this paper, it is
assumed that $H\in (\frac{1}{2}, 1)$.

This process was introduced by \cite{kol} and later
studied by \cite{MV}.  Its self-similar
and long-range dependence make this process a useful
driving noise in
 models arising in physics, telecommunication networks, finance and other fields.


 Consider a time interval $[0,T]$ with arbitrary fixed
horizon $T$ and let $\{\beta^H(t) , t \in [0, T ]\}$ the
one-dimensional fractional Brownian motion with Hurst parameter
$H\in(1/2,1)$. It is well known that  $\beta^H$ has the following
Wiener integral representation:
\begin{equation}\label{rep}
\beta^H(t) =\int_0^tK_H(t,s)d\beta(s),
 \end{equation}
where $\beta = \{\beta(t) :\; t\in [0,T]\}$ is a Wiener process,
and $K_H(t; s)$ is the kernel given by
$$K_H(t, s )=c_Hs^{\frac{1}{2}-H}\int_s^t (u-s)^{H-\frac{3}{2}}u^{H-\frac{1}{2}}du,
$$
for $t>s$, where $c_H=\sqrt{\frac{H(2H-1)}{\beta (2-2H,H-\frac{1}{2})}}$ and $\beta(,)$
 denotes the Beta function. We put $K_H(t, s ) =0$ if $t\leq s$.\\
We will denote by $\mathcal{H}$ the reproducing kernel Hilbert
space of the fBm. In fact $\mathcal{H}$ is the closure of the set of
indicator functions $\{1_{[0;t]},  t\in[0,T]\}$ with respect to
the scalar product
$$\langle 1_{[0,t]},1_{[0,s]}\rangle _{\mathcal{H}}=R_H(t , s).$$
The mapping $1_{[0,t]}\rightarrow \beta^H(t)$
 can be extended to an isometry between $\mathcal{H}$
and the first  Wiener chaos and we will denote by
$\beta^H(\varphi)$ the image of $\varphi$ by the previous
isometry.

We recall that for $\psi,\varphi \in \mathcal{H}$ their scalar
product in $\mathcal{H}$ is given by
$$\langle \psi,\varphi\rangle _{\mathcal{H}}=H(2H-1)\int_0^T\int_0^T\psi(s)\varphi(t)|t-s|^{2H-2}dsdt.
$$
Let us consider the operator $K_H^*$ from $\mathcal{H}$ to
$\mathbb{L}^2([0,T])$ defined by
$$(K_H^*\varphi)(s)=\int_s^T\varphi(r)\frac{\partial K}{\partial
r}(r,s)dr.
$$ We refer to \cite{nualart} for the proof of the fact
that $K_H^*$ is an isometry between $\mathcal{H}$ and
$L^2([0,T])$. Moreover for any $\varphi \in \mathcal{H}$, we have
$$\beta^H(\varphi)=\int_0^T(K_H^*\varphi)(t)d\beta(t).$$
It follows from \cite{nualart} that the elements of $\mathcal{H}$
may be not functions but distributions of negative order. In the
case $H>\frac{1}{2}$, the second partial derivative of the
covariance function
$$
\frac{\partial R_H}{\partial t\partial s}=\alpha_H|t-s|^{2H-2},
$$
where $\alpha_H=H(2H-2)$, is integrable, and we can write
\begin{equation}\label{r(t,t)}
R_H(t,s)=\alpha_H\int_0^t\int_0^s|u-v|^{2H-2}dudv.
\end{equation}

In order to obtain a space of functions contained in
$\mathcal{H}$, we consider the linear space $|\mathcal{H}|$
generated by the measurable functions $\psi$ such that
$$\|\psi \|^2_{|\mathcal{H}|}:= \alpha_H  \int_0^T \int_0^T|\psi(s)||\psi(t)| |s-t|^{2H-2}dsdt<\infty,
$$
where $\alpha_H = H(2H-1)$. The space $|\mathcal{H}|$ is a Banach
space with the norm  $\|\psi\|_{|\mathcal{H}|}$ and we have the
following inclusions (see \cite{nualart}).
\begin{lemma}\label{lem1}
$$\mathbb{L}^2([0,T])\subseteq \mathbb{L}^{1/H}([0,T])\subseteq |\mathcal{H}|\subseteq \mathcal{H},
$$
and for any $\varphi\in \mathbb{L}^2([0,T])$, we have
$$\|\psi\|^2_{|\mathcal{H}|}\leq 2HT^{2H-1}\int_0^T
|\psi(s)|^2ds.
$$
\end{lemma}
Let $X$ and $Y$ be two real, separable Hilbert spaces and let
$\mathcal{L}(Y,X)$ be the space of bounded linear operator from
$Y$ to $X$. For the sake of convenience, we shall use the same
notation to denote the norms in $X,Y$ and $\mathcal{L}(Y,X)$. Let
$Q\in \mathcal{L}(Y,Y)$ be an operator defined by $Qe_n=\lambda_n
e_n$ with finite trace
 $trQ=\sum_{n=1}^{\infty}\lambda_n<\infty$. where $\lambda_n \geq 0 \; (n=1,2...)$ are non-negative
  real numbers and $\{e_n\}\;(n=1,2...)$ is a complete orthonormal basis in $Y$.
 Let $B^H=(B^H(t))$ be  $Y-$ valued fbm on
  $(\Omega,\mathcal{F}, \mathbb{P})$ with covariance $Q$ as
 $$B^H(t)=B^H_Q(t)=\sum_{n=1}^{\infty}\sqrt{\lambda_n}e_n\beta_n^H(t),
 $$
 where $\beta_n^H$ are real, independent fBm's. This process is  Gaussian, it
 starts from $0$, has zero mean and covariance:
 $$E\langle B^H(t),x\rangle\langle B^H(s),y\rangle=R(s,t)\langle Q(x),y\rangle \;\; \mbox{for all}\; x,y \in Y \;\mbox {\,and}\;  t,s \in [0,T].
 $$
In order to define Wiener integrals with respect to the $Q$-fBm,
we introduce the space $\mathcal{L}_2^0:=\mathcal{L}_2^0(Y,X)$  of
all $Q$-Hilbert-Schmidt operators $\psi:Y\rightarrow X$. We recall
that $\psi \in \mathcal{L}(Y,X)$ is called a $Q$-Hilbert-Schmidt
operator, if
$$  \|\psi\|_{\mathcal{L}_2^0}^2:=\sum_{n=1}^{\infty}\|\sqrt{\lambda_n}\psi e_n\|^2 <\infty,
$$
and that the space $\mathcal{L}_2^0$ equipped with the inner
product
$\langle \varphi,\psi \rangle_{\mathcal{L}_2^0}=\sum_{n=1}^{\infty}\langle
\varphi e_n,\psi e_n\rangle$ is a separable Hilbert space.\\
Now, let $\phi(s);\,s\in [0,T]$ be a function with values in
$\mathcal{L}_2^0(Y,X)$, such that \\ $\sum_{n=1}^{\infty}\|K^*\phi
Q^{\frac{1}{2}}e_n\|_{\mathcal{L}_2^0}^2<\infty.$ The Wiener
integral of $\phi$ with respect to $B^H$ is defined by

\begin{equation}\label{int}
\int_0^t\phi(s)dB^H(s)=\sum_{n=1}^{\infty}\int_0^t
\sqrt{\lambda_n}\phi(s)e_nd\beta^H_n(s)=\sum_{n=1}^{\infty}\int_0^t
\sqrt{\lambda_n}(K_H^*(\phi e_n)(s)d\beta_n(s)
\end{equation}
where $\beta_n$ is the standard Brownian motion used to  present $\beta_n^H$ as in $(\ref{rep})$.\\
Now, we end this subsection by stating the following result which is
fundamental to prove our result.

\begin{lemma}\cite{boufoussi14}\label{lem2}
Suppose that $\sigma:[0,T]\rightarrow \mathcal{L}_2^0(Y,X)$
satisfies ${\sup_{t\in[0,T]}
\|\sigma(t)\|^2_{\mathcal{L}_2^0}<\infty}$, and Suppose that
$\{U(t,s),\, 0\leq s\leq t \leq T\}$ is an evolution system of
operators satisfying  $ \|U(t,s)\|\leq Me^{-\beta(t-s)},$  for some
constants $\beta>0$ and $M\geq 1$   for  all   $t\geq s. $ Then, we
have $$ \mathbb{E}\|\int_0^tU(t,s)\sigma(s)dB^H(s)\|^2\leq C
M^2t^{2H}
  (\sup_{t\in[0,T]} \|\sigma(t)\|_{\mathcal{L}_2^0})^2.$$
\end{lemma}

\begin{remark}
Thanks to Lemma \ref{lem2}, the stochastic integral $$
Z(t)=\int_{0}^{t}U(t,s)\sigma(s)dB^H(s),\qquad t\in[0,T],
$$ is
well-defined.
\end{remark}
\section{Controllability Result}

Henceforth we will assume that the family $\{A(t),\,t\in[0,T]\}$
 of linear operators generates an evolution system of operators $\{U(t,s),\, 0\leq s\leq t \leq
T\}$.  In this section we derive controllability conditions  for
time-dependent  neutral stochastic functional differential equations
with variable delays driven   by    a fractional Brownian motion in
a real separable Hilbert space. Before starting, we introduce the
concept of a mild solution of the problem (\ref{eq1})  and
controllability of
 neutral stochastic functional differential equation.

\begin{definition}
An $X$-valued  process $\{x(t),\;t\in[-\tau,T]\}$, is called  a
mild solution of equation (\ref{eq1}) if

\begin{itemize}
\item[$i)$] $x(.)\in \mathcal{C}([-\tau,T],\mathbb{L}^2(\Omega,X))$,
\item[$ii)$] $x(t)=\varphi(t), \, -\tau \leq t \leq 0$.
\item[$iii)$]For arbitrary $t \in [0,T]$, we have
\begin{equation}\label{eqmild2}
\begin{array}{ll}
x(t)&=U(t,0)(\varphi(0)+g(0,\varphi(-r(0))))-g(t,x(t-r(t)))\\ \\
&- \int_0^t AU(t,s)g(s,x(s-r(s)))ds +\int_0^t U(t,s)f(s,x(s-\rho
(s)))ds\\ \\
&+\int_0^t U(t,s)(Bu)(s)ds+\int_0^t U(t,s)\sigma(s)dB^H(s),\;\;\;
\mathbb{P}-a.s.
\end{array}
\end{equation}
\end{itemize}
\end{definition}
\begin{definition}
The  system (\ref{eq1}) is said to be controllable on the interval
$[-\tau,T]$, if for every initial stochastic process $\varphi$
defined $[-\tau,0]$, there exists a stochastic control $u\in
L^2([0,T], U)$ such that the mild solution $x(.)$ of  (\ref{eq1})
satisfies $x(T)=x_1$, where $x_1$ and $T$ are the preassigned
terminal state and time, respectively.
\end{definition}
We will study the problem  (\ref{eq1}) under   the following
assumptions:

\begin{itemize}
\item [$(\mathcal{H}.1)$]
\begin{itemize}
  \item [$i)$]The evolution family is exponentially
stable, that is, there exist two constants $\beta>0$ and $M\geq 1$
such that
$$
\|U(t,s)\|\leq Me^{-\beta(t-s)},\qquad     for\, all \quad t\geq
s,
$$
  \item [$ii)$] There exist a constant $M_*>0$\, such that
$$
\|A^{-1}(t)\| \leq M_*\qquad for \; all \quad t\in [0,T].
$$
\end{itemize}
\item [$(\mathcal{H}.2)$] The maps   $f, g:[0,T]\times X \rightarrow X
$ are continuous  functions and there  exist two positive
constants $C_1$ and $C_2$, such that for all $t\in[0,T]$ and
$x,y\in X$:
\begin{itemize}
  \item[$i)$] $\|f(t,x)-f(t,y)\|\vee \|g(t,x)-g(t,y)\|\leq C_1\|x-y\|$.
  \item[$ii)$]  $\|f(t,x)\|^2 \vee \|A^k(t)g(t,x)\|^2\leq C_2(1+\|x\|^2),\quad k=0,1.$
\end{itemize}
\item [$(\mathcal{H}.3)$]

\begin{itemize}
  \item [$i)$] There exists a positive  constant $L_*$ such that
  $L^*M_*<\frac{1}{\sqrt{6}}$, and
$$
\|A(t)g(t,x)-A(t)g(t,y)\|\leq L_*\|x-y\|,
$$
for all $t\in[0,T]$ and $x,y\in X$.
  \item [$ii)$] The function $g$ is continuous in the quadratic mean
  sense: for all $x(.)\in\mathcal{C}([0,T],L^2(\Omega,X))$,
  we have
  $$
\lim_{t\longrightarrow s}\E\|g(t,x(t))-g(s,x(s))\|^2=0.
  $$
\end{itemize}

\item [$(\mathcal{H}.4)$]

\begin{itemize}
  \item [$i)$] The map  $\sigma:[0,T]\longrightarrow \mathcal{L}^0_2(Y,X)$
is bounded, that is :  there exists a positive  constant $L$  such
that $\|\sigma (t)\|_{\mathcal{L}^0_2(Y,X)}\leq L$
uniformly in $t \in [0,T]$.\\
  \item [$ii)$] Moreover, we assume that the initial data $\varphi=\{\varphi(t): -\tau\leq t \leq 0 \}$ satisfies
   $\varphi \in \mathcal{C}([-\tau,0],\mathbb{L}^2(\Omega,X))$.\\
\end{itemize}

 \item [$(\mathcal{H}.5)$] The linear operator $W$ from $U$ into $X$
 defined by
 $$
Wu=\int_0^TU(T,s)Bu(s)ds
 $$
 has an inverse operator $W^{-1}$ that takes values in $L^2([0,T],U)\setminus ker
 W$, where $ker
 W=\{x\in L^2([0,T],U),  \; W x=0\}$ (see \cite{klam07}), and there
 exists finite
 positive constants $M_b,$ $M_w$ such that $\|B\|\leq M_b$ and $\|W^{-1}\|\leq M_w.$
\end{itemize}

The main result of this paper is given in the next theorem.
\begin{theorem}\label{jthm1}
Suppose that $(\mathcal{H}.1)-(\mathcal{H}.5)$ hold. Then,  the
system  (\ref{eq1}) is controllable  on $[-\tau,T]$.
\end{theorem}

\begin{proof}

Fix $T>0$ and let  $\mathcal{B}_T := \mathcal{C}([-\tau, T],
\mathbb{L}^2(\Omega, X)$ be the Banach space of all   continuous
functions from $[-\tau, T]$ into $\mathbb{L}^2(\Omega, X)$,
equipped  with the  supremum norm
$\|\xi\|_{\mathcal{B}_T}=\displaystyle\sup_{u \in
[-\tau,T]}\left(\mathbb{E} \|\xi (u)\|^2\right)^{1/2}$ and let us
consider the set
 $$S_T=\{x\in \mathcal{B}_T : x(s)=\varphi(s),\; \mbox {for} \;\;s \in [-\tau,0] \}.$$
 $S_T$ is a closed subset of $\mathcal{B}_T$ provided with the norm  $\|.\|_{\mathcal{B}_T}$.

Using the hypothesis $(\mathcal{H}.5)$  for an arbitrary function
$x(.)$, define the stochastic  control
$$
\begin{array}{lll}
  u(t) & =&W^{-1}\{x_1-U(T,0)(\varphi(0)+g(0,\varphi(-r(0))))+g(T,x(T-r(T)))\\ \\
   & +&\int_0^T AU(T,s)g(s,x(s-r(s)))ds -\int_0^T U(T,s)f(s,x(s-\rho (s))ds\\ \\
   & -& \int_0^T
   U(T,s)\sigma(s)dB^H(s).
\end{array}
$$
We will now show that using this control that the operator  $\psi$
on $S_T(\varphi)$ defined  by $\psi(x)(t)=\varphi(t)$ for $t\in
[-\tau,0]$
 and for $t\in [0,T]$
 $$
 \begin{array}{ll}
   \psi(x)(t)&=U(t,0)(\varphi(0)+g(0,\varphi(-r(0))))-g(t,x(t-r(t)))-\int_0^t
   U(t,s)A(s)g(s,x(s-r(s)))ds\\ \\
&+\int_0^t U(t,s)f(s,x(s-\rho
(s)))ds+\int_0^tU(t,\nu)\sigma(s)dB^H(s) \\
\\
&+\int_0^tU(t,\nu)BW^{-1}\{x_1-U(T,0)(\varphi(0)+g(0,\varphi(-r(0))))+g(T,x(T-r(T)))\\\\
    & +\int_0^T U(T,s)A(s)g(s,x(s-r(s)))ds -\int_0^T U(T,s)f(s,x(s-\rho (s))ds\\   \\
    & -\int_0^T
   U(T,s)\sigma(s)dB^H(s) \}d\nu,
 \end{array}
$$
has a fixed point. This fixed point is then a solution of
(\ref{eq1}). Clearly, $\psi(x)(T)=x_1$,  which  implies that the
system (\ref{eq1}) is controllable.\\
For better readability, we break the proof into sequence of steps.\\

 {\bf Step 1:} $\psi$ is well defined.  Let $x \in S_T(\varphi)$ and $t\in [0,T]$, we are going to show that each function $\psi(x)(.)$
 is continuous  on $[0,T]$ in the $\mathbb{L}^2(\Omega,X)$-sense.

Let $0 <t<T$  and $|h|$  be sufficiently small. Then for any fixed
$x\in S_T$, we have
$$
\begin{array}{lll}
\E\|\psi(x)(t&+&h)-\psi(x)(t)\|^2\leq
6\E\|(U(t+h,0)-U(t,0))(\varphi(0)+g(0,\varphi(-r(0))))\|^2\\ \\
&+&6\E\|g(t+h,x(t+h-r(t+h)))-g(t,x(t-r(t)))\|^2\\ \\
&+&6\E\|\int_0^{t+h} U(t+h,s)A(s)g(s,x(s-r(s))ds -\int_0^t
U(t,s)A(s)g(s,x(s-r(s))ds\|^2\\ \\
&+&6\E\|\int_0^{t+h} U(t+h,s)f(s,x(s-\rho (s)))ds-\int_0^t
U(t,s)f(s,x(s-\rho (s)))ds\|^2\\ \\
&+&6\E\|\int_0^{t+h}U(t+h,s)\sigma(s)dB^H(s)-\int_0^tU(t,s)\sigma(s)dB^H(s)\|^2\\
\\

&+&6\E \|
\int_0^{t+h}U(t+h,\nu)BW^{-1}\{x_1-U(T,0)(\varphi(0)+g(0,\varphi(-r(0))))\\
\\
&+&g(T,x(T-r(T)))  +  \int_0^T U(T,s)A(s)g(s,x(s-r(s)))ds \\ \\
&-&\int_0^T U(T,s)f(s,x(s-\rho (s))ds
     -\int_0^T
   U(T,s)\sigma(s)dB^H(s) \}d\nu\\ \\
&-&\int_0^{t}U(t,\nu)BW^{-1}\{x_1-U(T,0)(\varphi(0)+g(0,\varphi(-r(0))))\\
\\ &+&g(T,x(T-r(T)))   - \int_0^T
   U(T,s)\sigma(s)dB^H(s) \}d\nu\\
   \\
 &=& 6 \sum_{1\leq i \leq 6}\E\|I_i(t+h)-I_i(t)\|^2.
\end{array}
$$

From Definition \ref{d1}, we obtain
$$
\lim_{h\longrightarrow0}(U(t+h,0)-U(t,0))(\varphi(0)+g(0,\varphi(-r(0))))=0.
$$
From $(\mathcal{H}.1)$, we have
$$
\|(U(t+h,0)-U(t,0))(\varphi(0)+g(0,\varphi(-r(0))))\|\leq Me^{-\beta
t}(e^{-\beta h}+1)\|\varphi(0)+g(0,\varphi(-r(0)))\|\in L^2(\Omega).
$$
Then we conclude by the Lebesgue dominated theorem that
$$
\lim_{h\longrightarrow0}\E\|I_1(t+h)-I_1(t)\|^2=0.
$$
 Moreover, assumption $(\mathcal{H}.2)$ ensures that
$$
\lim_{h\longrightarrow0}\E\|I_2(t+h)-I_2(t)\|^2=0.
$$
To show that the third term $I_3(h)$ is continuous, we suppose
$h>0$ (similar calculus for $h<0$). We have

\begin{eqnarray*}
\|I_3(t+h)-I_3(t)\|&\leq & \left\|\int_0^t ( U(t+h,s)-U(t,s))A(s) g(s,x(s-r(s)))ds\right\|\\
&& +\left\|\int_t^{t+h} (U(t,s)g(s,x(s-r(s)))ds\right\|\\
&\leq & I_{31}(h)+I_{32}(h).
\end{eqnarray*}
By H\"older's inequality, we have
$$
\E\|I_{31}(h)\|\leq t\E\int_0^t \|
U(t+h,s)-U(t+h,s))A(s)g(s,x(s-r(s))\|^2ds.
$$
By Definition \ref{d1}, we obtain
$$
\lim_{h\longrightarrow0}( U(t+h,s)-U(t,s))A(s)g(s,x(s-r(s)))=0.
$$
From $(\mathcal{H}.1)$ and $(\mathcal{H}.2)$, we have
$$
\|U(t+h,s)-U(t,s))A(s)g(s,x(s-r(s))\|\leq
C_2Me^{-\beta(t-s)}(e^{-\beta h }+1) \|A(s)g(s,x(s-r(s))\|\in
L^2(\Omega).
$$
 Then we conclude by the Lebesgue dominated theorem that
$$
\lim_{h\longrightarrow0}\E\|I_{31}(h)\|^2=0.
$$
So, estimating as before. By  using $(\mathcal{H}.1)$ and
$(\mathcal{H}.2)$, we get
$$
\E\|I_{32}(h)\|^2\leq \frac{M^2C_2(1-e^{-2\beta
h})}{2\beta}\int_t^{t+h}(1+\E\|x(s-r(s))\|^2)ds.
$$
Thus,
$$
\lim_{h\longrightarrow0}\E\|I_{32}(h)\|^2=0.
$$

 For the fourth term $I_4(h)$, we suppose $h>0$
(similar calculus for $h<0$). We have
\begin{eqnarray*}
\|I_4(t+h)-I_4(t)\|&\leq & \left\|\int_0^t ( U(t+h,s)-U(t,s))f(s,x(s-\rho(s)))ds\right\|\\
&& +\left\|\int_t^{t+h} (U(t,s)f(s,x(s-\rho(s)))ds\right\|\\
&\leq & I_{41}(h)+I_{42}(h).
\end{eqnarray*}

By H\"older's inequality, we have
$$
\E\|I_{41}(h)\|\leq t\E\int_0^t \|
U(t+h,s)-U(t,s))f(s,x(s-\rho(s))\|^2ds.
$$

Again exploiting properties of Definition \ref{d1}, we obtain
$$
\lim_{h\longrightarrow0}( U(t+h,s)-U(t,s))f(s,x(s-\rho(s)))=0,
$$
and
$$
\|U(t+h,s)-U(t,s))f(s,x(s-\rho(s))\|\leq
Me^{-\beta(t-s)}(e^{-\beta h }+1)\|f(s,x(s-\rho(s))\|\in
L^2(\Omega).
$$
 Then we conclude by the Lebesgue dominated theorem that
$$
\lim_{h\longrightarrow0}\E\|I_{41}(h)\|^2=0.
$$

 On the  other hand, by $(\mathcal{H}.1)$ ,  $(\mathcal{H}.2)$, and the H\"older's inequality, we have
\begin{eqnarray*}
\E\|I_{42}(h)\|\leq\frac{M^2C_2(1-e^{-2\beta
h})}{2\beta}\int_t^{t+h} (1+\E\|x(s-\rho(s))\|^2)ds.
\end{eqnarray*}
Thus
 $$
 \lim_{h\rightarrow 0}I_{42}(h)=0.$$

Now, for the term $I_5(h)$, we have
\begin{eqnarray*}
\|I_5(t+h)-I_5(t)\|&\leq & \|\int_0^t (U(t+h,s)-U(t,s)\sigma(s)dB^H(s)\|\\
&+&\|\int_t^{t+h} U(t+h,s)\sigma(s)dB^H(s)\|\\
&\leq & I_{51}(h)+I_{52}(h).
\end{eqnarray*}

By   Lemma \ref{lem2}, we get that
\begin{eqnarray*}
E\|I_{51}(h)\|^2&\leq &2Ht^{2H-1}\int_0^t \|[U(t+h,s)-U(t,s)]\sigma(s)\|_{\mathcal{L}_2^0}^2ds.\\
\end{eqnarray*}
 Since
 $$
 \displaystyle\lim_{h\rightarrow 0} \|[U(t+h,s)-U(t,s)] \sigma(s)\|_{\mathcal{L}_2^0}^2=0
 $$
  and
 $$\|(U(t+h,s)-U(t,s) \sigma(s)\|_{\mathcal{L}_2^0}\leq MLe^{-\beta (t-s)}e^{-\beta h+1}\in \mathbb{L}^1([0,T],\, ds),$$
 we conclude, by the dominated convergence theorem that,
 $$ \lim_{h\rightarrow 0}\mathbb{E}\|I_{51}(h)\|^2=0.  $$
 Again by Lemma \ref{lem2}, we get that

$$
\mathbb{E}\|I_{52}(h)\|^2\leq  \frac{2Ht^{2H-1}LM^2(1-e^{-2\beta h
})}{2\beta}.
$$
Thus,
 $$ \lim_{h\rightarrow 0}\mathbb{E}|I_{52}(h)|^2=0.  $$

For the estimation of term $I_6$, we have

$$
\begin{array}{ll}
 \E\|I_{6}(h)\|^2 &\leq 2\E
\|
\int_t^{t+h}U(t+h,\nu)BW^{-1}\{x_1-U(T,0)(\varphi(0)+g(0,\varphi(-r(0))))\\\\
&+ g(T,x(T-r(T)))+ \int_0^T U(T,s)A(s)g(s,x(s-r(s)))ds\\\\
    &-\int_0^T U(T,s)f(s,x(s-\rho (s))ds -\int_0^T
   U(T,s)\sigma(s)dB^H(s)\|\\\\
   &+2\E\|\int_0^{t}(U(t+h,\nu)-U(t,\nu))BW^{-1}\{x_1-U(T,0)(\varphi(0)+g(0,\varphi(-r(0))))\\\\
   &+g(T,x(T-r(T)))    + \int_0^T U(T,s)A(s)g(s,x(s-r(s)))ds \\ \\
   &-\int_0^T U(T,s)f(s,x(s-\rho (s))ds-\int_0^T
   U(T,s)\sigma(s)dB^H(s)\}d\nu\|
   \\\\
   &\leq 2[\E\|I_{6,1}(h)\|^2+\E\|I_{6,2}(h)\|^2].
\end{array}
$$

 Let's first deal with $I_{6,1}(h)$, it follows from the conditions
 $(\mathcal{H}.1)-(\mathcal{H}.5)$ that

$$
\begin{array}{ll}
  \E\|I_{6,1}(h)\|^2 &\leq6 M^2M_b^2M_w^2\int_t^{t+h}\{\E\|x_1\|^2+M^2\E\|\varphi(0)+g(0,\varphi(-r(0)))\|^2\\\\
  &+M_*^2C_2T(1+\sup_{s\in[-\tau,T]}\E\|x(s)\|^2) +M^2TC_2(1+\sup_{s\in[-\tau,T]}\E\|x(s)\|^2)\\\\
  &+M^2TC_2(1+\sup_{s\in[-\tau,T]}\E\|x(s)\|^2)+2M^2HT^{2H-1}\int_0^T\|\sigma(s)\|_{\mathcal{L}_2^0}^2ds\}d\nu.
\end{array}
$$

It results that
$$
 \lim_{h\rightarrow 0}\mathbb{E}||I_{6,1}(h)||^2=0.
$$

In a similar way, we have
$$
\begin{array}{ll}
  \E\|I_{6,2}(h)\|^2 & \leq6 M_b^2M_w^2\int_0^{t}\|U(t+h,\nu)-U(t,\nu)\|^2\{\E\|x_1\|^2\\ \\
  &+M^2\E\|\varphi(0)+g(0,\varphi(-r(0)))\|^2
  +M_*^2C_2(1+\E\|x\|^2)\\\\
   &+ M^2T^2   C_2(1+\E\|x\|^2)+M^2
   T^2C_2(1+\E\|x\|^2)\\\\
   &+2M^2HT^{2H-1}\int_0^T\|\sigma(s)\|_{\mathcal{L}_2^0}^2ds\}d\nu.
\end{array}
$$
Since
$$
\begin{array}{ll}
 & \|U(t+h,\nu)-U(t,\nu)\|^2\{\E\|x_1\|^2+M^2\E\|\varphi(0)+g(0,\varphi(-r(0)))\|^2
  +M_*^2C_4(1+\sup_{s\in[-\tau,T]}\E\|x(s)\|^2)\\ \\
  &+ M^2 T^2  C_2(1+\sup_{s\in[-\tau,T]}\E\|x(s)\|^2)+M^2T^2C_2(1+\sup_{s\in[-\tau,T]}\E\|x(s)\|^2)\\ \\
  &+2M^2HT^{2H-1}\int_0^T\|\sigma(s)\|_{\mathcal{L}_2^0}^2ds \}\\ \\
 &  \leq  4M^2   \{\E\|x_1\|^2+M^2\E\|\varphi(0)+g(0,\varphi(-r(0)))\|^2\\ \\
  &+M_*^2C_2(1+\sup_{s\in[-\tau,T]}\E\|x(s)\|^2)+ 2M^2T^2   C_2(1+\sup_{s\in[-\tau,T]}\E\|x(s)\|^2)\\ \\
  &+2M^2HT^{2H-1}\int_0^T\|\sigma(s)\|_{\mathcal{L}_2^0}^2ds\}  \in
  L^1([0,T],ds]),
\end{array}
$$
we conclude, by the dominated  convergence  theorem that,
$$
\lim_{h\rightarrow 0}\mathbb{E}||I_{6,2}(h)||^2=0.
$$

The above arguments show that $\displaystyle\lim_{h\rightarrow
0}\mathbb{E}\|\psi(x)(t+h)-\psi(x)(t)\|^2=0$.  Hence, we conclude
that  the function  $t \rightarrow \psi(x)(t)$ is continuous on
$[0,T]$ in the $\mathbb{L}^2$-sense.

{\bf Step 2:}    Now, we are going to show that $\psi$ is a
contraction mapping in $S_{T_1}(\varphi)$ with some $T_1\leq T$ to
be specified later.
 Let $x,y\in S_T(\varphi)$,    then for   any fixed  $t\in [0,T]$, we have
$$
\begin{array}{ll}
\E&\|\psi(x)(t)-\psi(y)(t)\|^2\\ \\
&\leq6\|A(t)^{-1}\|^2\E\|A(t)g(t,x(t-r(t)))-A(t)g(t,y(t-r(t)))\|^2\\ \\
&+6\E\|\int_0^t U(t,s)A(s)(g(s,x(s-r(s)))-g(s,y(s-r(s))))ds\|^2\\ \\
&+6\E\|\int_0^tU(t,s)(f(s,x(s-\rho(s)))-f(s,y(s-\rho(s))))ds\|^2\\ \\
&+6\E\|
\int_0^{t}U(t,\nu)BW^{-1}[g(T,x(T-r(T)))-g(T,y(T-r(T)))]d\nu\|^2\\
\\
&+6\E\| \int_0^{t}U(t,\nu)BW^{-1}\int_0^T
U(T,s)A(s)[g(s,x(s-r(s)))-g(s,y(s-r(s)))]ds]d\nu\|^2
\\\\
 &+ 6\E\|\int_0^{t}U(t,\nu)BW^{-1}\int_0^T U(T,s)[f(s,x(s-\rho
 (s)))-f(s,y(s-\rho (s)))]dsd\nu\|^2.
\end{array}
$$
By assumptions   combined with H\"older's inequality, we get that
$$
\begin{array}{ll}
\mathbb{E}\|\psi(x)(t)-\psi(y)(t)\|^2 &\leq  6 L_*^2M_*^2
\sup_{s\in[-\tau,t]}\mathbb{E}\|x(t-r)-y(t-r)\|^2\\ \\
 &+ 6
M^2L_*^2\frac{1-e^{-2\beta t}}{2\beta}t
\sup_{s\in[-\tau,t]} \mathbb{E}\|x(s)-y(s)\|^2\\\\
&+6M^2C_1^2\frac{1-e^{-2\beta t}}{2\beta}t
\sup_{s\in[-\tau,t]} \mathbb{E}\|x(s)-y(s)\|^2\\\\
&+6\bf{t}
M^2M_b^2M_w^2[C_1^2\E\|x(T-r(T))-y(T-r(T))\|^2\\\\
&+L_*^2M^2T^2\sup_{s\in[-\tau,t]} \mathbb{E}\|x(s-r(s))-y(s-r(s))\|^2\\\\
&+T^2M^2C_1^2  \sup_{s\in[-\tau,t]} \mathbb{E}\|x(s)-y(s)\|^2.\\\\
\end{array}
$$

Hence
$$\sup_{s\in[-\tau,T]}\mathbb{E}\|\psi(x)(s)-\psi(y)(s)\|^2\leq
\gamma(t) \sup_{s\in[-\tau,T]} \mathbb{E}\|x(s)-y(s)\|^2,
$$
where


$$
\begin{array}{ll}
  \gamma(t)=&6[\|L_*^2M_*^2+M^2L_*^2\frac{1-e^{-2\beta t}}{2\beta}t+M^2C_1^2\frac{1-e^{-2\beta
  t}}{2\beta}t
  \\\\
   & +\bf{t} M^2M_b^2M_w^2
(C_1^2+L_*^2M^2T^2+T^2M^2C_1^2].
\end{array}
$$

By condition   $(\mathcal{H}.3)$, we have $\gamma(0)=6 L_*^2M_*^2
<1$. Then there exists $0<T_1\leq T $ such that $0<\gamma(T_1)<1$
and $\psi$ is a contraction mapping on $S_{T_1}$ and therefore has a
unique fixed point, which is a mild solution of equation (\ref{eq1})
on $[-\tau,T_1]$. This procedure can be repeated in order to extend
the solution to the entire interval $[-\tau,T]$ in finitely  many
steps. Clearly, $(\psi x)(T)=x_1$ which implies that the system
(\ref{eq1}) is controllable on $[-\tau,T]$.    This completes the
proof.
\end{proof}

\section{An illustrative  Example}
In recent years, the interest in neutral systems has been growing
rapidly due to their successful applications in practical fields
such as physics, chemical technology, bioengineering, and electrical
networks.  We consider the following stochastic partial neutral
functional differential equation with finite   delays $\tau_1$ and
$\tau_2$ $(0\leq \tau_i\leq\tau<\infty,\; i=1,2)$:
\begin{equation}\label{pe}
\left\{\begin{array}{lll}
   d\left[u(t,\zeta)+g_1(t,u(t-\tau_1,\zeta))\right]&= &
    [\frac{\partial^2}{\partial^2\zeta}u(t,\zeta)+b(t,\zeta)u(t,\zeta)+f_1(t,
    u(t-\tau_2,\zeta))\\ \\&+&v(t,\xi)]dt
  + \sigma(t) dB^H(t),\,0\leq t  \leq T,\,0\leq \zeta
   \leq\pi,\\\\
 u(t,0)=u(t,\pi) = 0,&& \hspace{-1cm}0\leq t  \leq T, \\ \\
   u(t,\zeta)= \varphi(t,\zeta),&&\hspace{-2.5cm} t\in[-\tau,0],\; 0\leq \zeta
   \leq\pi,
  \end{array}\right.
\end{equation} where $B^H$ is a fractional Brownian motion,
$b(t,\zeta)$ is a continuous function and is uniformly H\"older
continuous in $t$,  $f_1$, $g_1:\R^+\times\R\longrightarrow\R$ are
continuous
functions.\\
To study this system, we consider the space $X=L^2([0,\pi])$ and the
operator $A:D(A)\subset X\longrightarrow X$ given  by $Ay=y{''}$
with
$$
D(A)=\{y\in X: y''\in X,\quad y(0))=y(\pi)=0\}.
$$
 It is well known that $A$ is the infinitesimal generator of an
analytic semigroup \{$T(t)\}_{t\geq 0}$ on $X$. Furthermore, $A$ has
discrete spectrum with eigenvalues $-n^2,\, n\in \N$
 and the corresponding normalized eigenfunctions given by
 $$
e_n:=\sqrt{\frac{2}{\pi}}\sin nx,\; n=1,2,....
$$
In addition  $(e_n)_{n\in\N}$ is  a complete orthonormal basis in
$X$ and $$ T(t)x=\sum_{n=1}^{\infty}e^{-n^2t}<x,e_n>e_n,
$$
for $x\in X$ and $t\geq0$.\\

Now,   we define an   operator $A(t):D(A)\subset X\longrightarrow X$
by
$$
A(t)x(\zeta)=Ax(\zeta)+b(t,\zeta)x(\zeta).
$$
By assuming that $b(.,.)$ is continuous and that
$b(t,\zeta)\leq-\gamma$ $(\gamma>0)$ for every $t\in \R$,
$\zeta\in[0,\pi]$,  it follows that the system
$$
  \left\{
\begin{array}{ll}
u'(t)&=  A(t)u(t), \quad t\geq s,   \\
u(s) &=  x\in X , \\
\end{array}
\right.
$$
has an associated evolution family given by
$$
U(t,s)x(\zeta)=\left[T(t-s)\exp(\int_s^t
b(\tau,\zeta))d\tau)x\right](\zeta).
$$
From this expression, it follows that $U(t,s)$ is a compact linear
operator and that for every $s,t\in[0,T]$ with $t>s$
$$
\|U(t,s)\|\leq e^{-(\gamma+1)(t-s)}
$$
In addition, $A(t)$ satisfies the assumption $\mathcal{H}_1$ (see
\cite{baghli, ren2009}).

To rewrite the initial-boundary value problem (\ref{pe}) in the
abstract form we assume the following:

\begin{itemize}
\item [$i)$]  $B: U\longrightarrow X$ is a bounded linear operator
  defined by
  $$
Bu(t)(\xi)=v(t,\xi),\;0\leq \xi \leq\pi, \, u\in L^2([0,T], U).
  $$
  \item[$ii)$]  The operator $W:L^2([0,T], U)\longrightarrow X $ defined by
  $$
W u=\int_0^TS(T-s)v(t,\xi)ds
  $$
  has an inverse $W^{-1}$  and satisfies condition
  $(\mathcal{H}.5)$. For the construction of the operator $W$ and
  its inverse, see \cite{qui85}.

  \item [$iii)$] The substitution operator $f:[0,T]\times X\longrightarrow
  X$ defined by $f(t,u)(.)=f_1(t,u(.))$ is continuous and we impose
  suitable conditions on $f_1$ to verify assumption  $\mathcal{H}_2$.
  \item [$iv)$] The substitution operator $g:[0,T]\times X\longrightarrow
  X$ defined by $g(t,u)(.)=g_1(t,u(.))$ is continuous and we impose
  suitable conditions on $g_1$ to verify assumptions  $\mathcal{H}_2$ and $\mathcal{H}_3$.
\end{itemize}
If we put
\begin{equation}
 \left\{\begin{array}{ll}
x(t)(\zeta)=x(t,\zeta),\;t\in[0,T],\,\zeta\in[0,\pi]\\
x(t,\zeta)= \varphi(t,\zeta), \; t\in[-\tau,0],\; \zeta\in[0,\pi],
 \end{array}\right.
\end{equation}

then,  the problem (\ref{pe}) can be written in the abstract form
{\small \begin{eqnarray*}
 \left\{\begin{array}{lll}
d[x(t)+g(t,x(t-r(t)))]=[A(t)x(t)+f(t,x(t-\rho(t))]dt+\sigma
(t)dB^H(t)   ,\;0\leq t \leq T,\nonumber\\
x(t)=\varphi(t) ,\;-\tau \leq t \leq 0.
\end{array}\right.
\end{eqnarray*}
}

Furthermore, if we  assume that the initial data
$\varphi=\{\varphi(t): -\tau\leq t \leq 0 \}$ satisfies   $\varphi
\in \mathcal{C}([-\tau,0],\mathbb{L}^2(\Omega,X))$, thus all the
assumptions of Theorem \ref{jthm1} are fulfilled. Therefore, we
conclude that the system (\ref{pe}) is controllable on $[-\tau,T]$.



\end{document}